\documentclass[11pt]{amsart}
\usepackage{setspace}
\usepackage{amsmath,amssymb,mathrsfs,color}
\usepackage{verbatim}
\allowdisplaybreaks
\usepackage{hyperref}
\hypersetup{hypertex=true,
	colorlinks=true,
	linkcolor=blue,
	anchorcolor=blue,
	citecolor=blue}
\allowdisplaybreaks

\def\d{{\rm{d}}}

\def\N{{\mathbb N}}
\def\Z{{\mathbb Z}}
\def\R{{\mathbb R}}

\def\P{{\mathbb P}}

\newtheorem{lemma}{Lemma}[section]
\newtheorem{theorem}[lemma]{Theorem}
\newtheorem{remark}[lemma]{Remark}
\newtheorem{prop}[lemma]{Proposition}

\newtheorem{definition}[lemma]{Definition}

\allowdisplaybreaks
\numberwithin{equation}{section}

\newcommand{\enabstractname}{Abstract}
\newcommand{\frabstractname}{摘要}

\allowdisplaybreaks
\begin{document}
\title[Continuity of measure-theoretic entropy for SDEs]
{Continuity of measure-theoretic entropy for stochastic differential equations}

\author{Zhenxin Liu}
\address{Z. Liu: School of Mathematical Sciences, Dalian University of Technology, Dalian 116024, P. R. China}
\email{zxliu@dlut.edu.cn}

\author{Lixin Zhang}
\address{L. Zhang (Corresponding author): School of Mathematical Sciences, Dalian University of Technology, Dalian 116024, P. R. China}
\email{lixinzhang8@hotmail.com}

\date{August 1, 2026}
\subjclass[2020]{28D20, 37A10, 37A35, 60H10.}
\keywords{Measure-theoretic entropy; Continuity of entropy; Stochastic differential equation}

\begin{abstract}
For stochastic differential equations (SDEs), we establish a relationship between the measure-theoretic entropy of the stochastic flow and the rate of volume growth of stable submanifolds under iteration. By combining the result of Kifer and Yomdin \cite{Kifer-Yomdin}, we show that under a suitable integrability condition, for systems with $C^{\infty}$ coefficients, the measure-theoretic entropy is upper semicontinuous with respect to the coefficients of SDEs.
\end{abstract}

\maketitle
\section{Introduction}
Measure-theoretic entropy is a fundamental quantity for quantifying the complexity and chaotic behavior in measure-preserving dynamical systems. It describes the exponential rate at which typical orbits diverge over time. Moreover, it is invariant under measure-theoretic conjugacy and depends only on the underlying dynamical structure, making it an effective tool for distinguishing dynamical systems that are inequivalent in the measure-theoretic sense.

The continuity of entropy is crucial for assessing the structural stability of a system under perturbations. This issue is particularly important for random dynamical systems, where inherent noise and modeling uncertainties naturally lead to perturbations of the underlying system. In such a setting, it is therefore desirable that the entropy remains a well-defined dynamical invariant.

Consider a probability space $(X,\mathcal{B},m)$ and a finite measurable partition $\xi=\left\{A_{1},A_{2},\cdots,A_{k}\right\}$ of $X$. The measure-theoretic entropy of $\xi$ is defined by
\begin{align}
H_{m}(\xi):=-\sum_{i=1}^{k}m(A_{i})\ln m(A_{i}).\notag
\end{align}
Furthermore, let $T:X\rightarrow X$ be a measure-preserving transformation with respect to $m$.
The measure-theoretic entropy of $T$ with respect to $m$ is defined by
\begin{align}
h_{m}(T):=\sup_{\xi\in\mathcal{A}}\lim\limits_{n\rightarrow\infty}\dfrac{1}{n}H_{m}(\bigvee_{i=1}^{n}T^{-i}\xi),\notag
\end{align}
where $\mathcal{A}$ is the set containing all finite measurable partitions of $X$.

Measure-theoretic entropy was introduced by Kolmogorov \cite{Kolmogorov} as an invariant in ergodic theory to quantify the complexity of dynamical systems. Subsequently, Sinai \cite{Sinai} transformed the measure-theoretic entropy from a purely theoretical definition into a computable quantity. In the setting of smooth dynamical systems, Ruelle \cite{Ruelle} established a fundamental inequality demonstrating that the measure-theoretic entropy is bounded above by the sum of all positive Lyapunov exponents. Under appropriate regularity and absolute continuity assumptions, Pesin \cite{Pesin} showed that equality holds.

In general, the entropy map $(T,m)\rightarrow h_{m}(T)$ is not lower semicontinuous, even in the uniformly hyperbolic case. It is straightforward to construct a sequence of periodic systems, each with zero entropy, that converges to a hyperbolic system with positive entropy.
However, Newhouse \cite{Newhouse-1989} proved that for $C^{\infty}$ diffeomorphisms on compact manifolds, the measure-theoretic entropy $h_{m}(T)$ is upper semicontinuous. In the more general setting of $C^{r}$
diffeomorphisms ($r\geq 1$), even upper semicontinuity may fail. Counterexamples were constructed in dimension four by Misiurewicz \cite{Misiurewicz} and in dimension two by Buzzi \cite{Buzzi}. However, even in the absence of upper semicontinuity, the quantity
\begin{align}
\limsup\limits_{(T_{k},m_{k})\rightarrow(T,m)}h_{m_{k}}(T_{k})-h_{m}(T),\qquad k\in \N^{+}\notag
\end{align}
can still be controlled from above via Yomdin theory \cite{Yomdin}.

Consider the following $d$-dimensional SDEs
\begin{align}\label{equ.1.1}
\d X=A(X)\d t+B(X)\d W,\quad X(0)=x\in\R^{d}
\end{align}
and
\begin{equation}\label{equ.1.2}
\d X=A_{k}(X)\d t+B_{k}(X)\d W, \quad X_k(0)=x, \quad k\in\N^{+},
\end{equation}
where $W$ is a $1$-dimensional Brownian motion and $A,B,A_{k},B_{k}: \R^{d}\rightarrow\R^{d}$. Let $\left\{\Phi(x,t,\omega)\right\}_{t\geq 0}$ and $\left\{\Phi_{k}(x,t,\omega)\right\}_{t\geq 0}$ be stochastic flows generated by equations \eqref{equ.1.1} and \eqref{equ.1.2} with invariant measures $\mu$ and $\mu_{k}$, respectively. The time-$t_{0}$-map ($t_{0}\in\R^{+}$) measure-theoretic entropy of $\left\{\Phi(x,t,\omega)\right\}_{t\geq 0}$ is defined as
$$
h_{\mu}^{t_{0}}(\Phi):=\sup_{\xi\in\mathcal{A}}\lim\limits_{n\rightarrow\infty}\frac{1}{n}\int_{\Omega}H_{\mu}(\bigvee_{i=0}^{n-1}\Phi^{-1}(it_{0},\omega)(\xi))\d\P(\omega),
$$
where the supremum is taken over the set of all finite measurable partitions of $\R^{d}$.
It is standard to consider the associated time-1-map measure-theoretic entropy, denoted by $h_{\mu}(\Phi):=h_{\mu}^{1}(\Phi)$.

Arnold \cite{Arnold-system}, Kifer \cite{Kifer}, and other researchers rigorously introduced the measure-theoretic entropy for random dynamical systems. Kifer \cite{Kifer} proved the Ruelle inequality for independent and identically distributed random transformations. For other versions, refer to Liu and Qian \cite{Liupeidong} and Bahnm$\ddot{\text{u}}$ller and Bogensch$\ddot{\text{u}}$tz \cite{Bahnmuller}. Ledrappier and Young \cite{Ledrappier} proved the random version of the Pesin entropy formula. Kifer and Yomdin \cite{Kifer-Yomdin} established a random version of Yomdin's theory \cite{Yomdin}.

In this paper, we investigate the limiting behavior of measure-theoretic entropy for SDEs, focusing on whether
$$
h_{\mu_{k}}(\Phi_{k})\rightarrow h_{\mu}(\Phi)
$$
holds when $A_{k}\rightarrow A$ and $B_{k}\rightarrow B$. We show that, under a suitable integrability condition, if the coefficients are $C^{\infty}$ functions and converge under the $C^{0}$ topology, then the measure-theoretic entropy is upper semicontinuous. Our result extends the upper semicontinuity result of  Newhouse \cite{Newhouse-1989} to the SDE setting. As a corollary, for SDEs driven by additive noise, we provide easily verifiable conditions ensuring the upper semicontinuity of the measure-theoretic entropy.

The paper is organized as follows. In section \ref{sec.2}, we recall some preliminaries related to SDEs and random dynamical systems. In section \ref{sec.3}, we establish the upper semicontinuity of the measure-theoretic entropy for SDEs (see Theorem~\ref{the.3.1}). In Section \ref{sec.4}, we present the continuity of the entropy under the Lyapunov function condition (see Theorem~\ref{the.4.2}) and apply our results to SDEs driven by additive noise (see Theorem~\ref{the.4.1}).

\section{Preliminaries}\label{sec.2}
Firstly, let us introduce several norms. For any vector $v=(a_{1},a_{2},\dots,a_{d})\in \R^{d}$, we denote the Euclidean norm by $\vert v\vert:=\sqrt{\sum_{i=1}^{d}\vert a_{i}\vert^{2}}$. Given a matrix  $A=(v_{1},v_{2},\dots,v_{d})\in \R^{d\times d}$, the matrix norm $\vert A\vert$  is defined as $\vert A\vert:=\max\limits_{1\leq i\leq d}\left\{\vert v_{i}\vert\right\}$.  Furthermore, the operator norm $\Vert A\Vert$ is given by
$$
\Vert A\Vert:=\sup_{\substack{v\in \R^{d} \\ v\neq 0}}\frac{\vert Av\vert}{\vert v\vert}.
$$
Note that there exists a constant $C$ such that for any $A\in \R^{d\times d}$,
\begin{equation}
C^{-1}\vert A\vert\leq\Vert A\Vert\leq C\vert A\vert.\notag
\end{equation}
We use $C^{r}(\R^{d})(0\leq r\leq\infty)$ to denote the space of functions with continuous $r$-th order derivatives on $\R^{d}$ and $C_{b}^{r}(\R^{d})$ the space of functions with continuous and bounded $i$-th ($0\le i \le r$) order derivatives. For any $\psi\in C_{b}^{r}(\R^{d})$, the norm is given by
$$
\Vert\psi\Vert_{C^{r}}:=\max\left\{\sup_{x\in \R^{d}}\left\{\vert\psi(x)\vert\right\}, \sup_{1\leq i\leq r}\sup_{x\in \R^{d}}\left\{\vert D ^{i}\psi(x)\vert\right\}\right\}.
$$
A function $f$ on $\R^{d}$ is called $C^{1,\alpha}$ ($0<\alpha\leq1$) if it is $1$-th continuously differentiable and the $1$-th derivatives
are $\alpha$-H$\ddot{\mathrm{o}}$lder continuous.

Consider the SDE \eqref{equ.1.1}. It is well-known that when $A,B$ are  Lipschitz and satisfy the linear growth condition, \eqref{equ.1.1} admits a unique solution $\left\{\Phi(x,t,\omega)\right\}_{t\geq 0,\omega\in\Omega}$; see e.g. Arnold \cite[p. 105]{Arnold-equation}.
Kunita \cite[p. 209]{Kunita} introduced the concept of stochastic flows,  which capture the ``time-variable'' orbital motion. As we know, the solution of SDE \eqref{equ.1.1} corresponds to a one-parameter stochastic flow (see the following Definition \ref{def.2.1}).

To prevent an overuse of notation, we adopt the following convention throughout this paper. For the solution $\{\Phi(x,t,\omega)\}_{t\geq 0,\omega\in\Omega}$ of \eqref{equ.1.1}, we also write $\{\Phi(x,t)\}_{t\geq 0}$ for simplicity as normal. If the dependence on $x$ is omitted, it represents a one-parameter stochastic flow $\{\Phi(t,\omega)\}_{t\ge0,\omega\in\Omega}$ with $\Phi(t,\omega)(\cdot)=\Phi(\cdot,t,\omega)$. If the dependence on $t$ is omitted, it refers to the iteration of $x$ under the time-$1$-map, i.e. $\Phi(x,\omega)=\Phi(x,1,\omega)$. If both $x$ and $t$ are omitted, then it denotes the random diffeomorphism generated by the time-$1$-map, i.e. $\Phi(\omega)(\cdot)=\Phi(\cdot,1,\omega)$. The same convention also applies to $\partial\Phi$ and $D\Phi$, similarly for the the solution $\Phi_k$ of \eqref{equ.1.2}.

\begin{definition}[Stochastic flow]\label{def.2.1}\rm
A {\em  one-parameter stochastic flow of homeomorphisms} on the probability space $(\Omega,\mathcal{F},\P)$ with the state space $\R^{d}$ is a
family of continuous maps $\left\{\Phi(t,\omega)| t\geq 0, \omega\in\Omega \right\}$ which satisfies the following properties for any $\omega$ from a subset $\Omega'\subset\Omega$ of full $\P$-measure:
\begin{enumerate}
\item[1.]
$\Phi(s+t,\omega)=\Phi(t,\theta(s)\omega)\circ\Phi(s,\omega)$ holds for all $s,t\geq 0$.
\item[2.]
The map $\Phi(t,\omega):\R^{d}\rightarrow \R^{d}$ is a  homeomorphism for all $t\geq 0$.
\end{enumerate}
\end{definition}

We can get more smoothness of the solution under additional smoothness assumption of the coefficients. We denote the Jacobian matrix of solution  $\left\{\Phi(x,t)\right\}_{t\geq 0}$ as $D\Phi(x,t)=(\partial_{1} \Phi(x,t),\partial_{2} \Phi(x,t),\cdots,\partial_{d}\Phi(x,t))$ where the partial derivatives $\partial_{l} \Phi(x,t)$, $1\leq l\leq d$ have the following conclusion (see e.g. Kunita \cite[p. 218]{Kunita}):

\begin{prop}\label{lem.2.1}
Suppose that coefficients $A,B$ in SDE \eqref{equ.1.1} are $C^{1,\alpha}$ functions for some $0<\alpha\leq1$, and $\vert DA\vert,\vert DB\vert$ are bounded by a constant $L$. Then the solution $\left\{\Phi(x,t)\right\}_{t\geq 0}$ is a $C^{1,\alpha}$ function of $x$ and for $1\leq l\leq d$ the derivative $\partial_{l}\Phi(x,t)=\frac{\partial \Phi(x,t)}{\partial x_{l}}$ satisfies the equation
\begin{equation}
\partial_{l} \Phi(x,t)=e_{l}+\int_{0}^{t}DA(\Phi(x,s))\partial_{l} \Phi(x,s)\d s+\int_{0}^{t}DB(\Phi(x,s))\partial_{l} \Phi(x,s)\d W \quad a.s.,\notag
\end{equation}
where $DA(x),DB(x)$ are the Jacobian matrices of $A(x),B(x)$ and $e_{l}$ is the unit vector $(0,\ldots,0,$ $1,0,\ldots,0)$ with $1$ as the $l$-th component.
\end{prop}

As discussed in Kunita \cite[p. 228]{Kunita}, the inverse $(D\Phi)^{-1}(x,t)$ of the Jacobian $D\Phi(x,t)$ satisfies the following equation
\begin{align*}
(D \Phi)^{-1}(x,t)=I-&\int_{0}^{t}DA(\Phi(x,s))(D \Phi)^{-1}(x,s)\d s-\int_{0}^{t}DB(\Phi(x,s))(D \Phi)^{-1}(x,s)\d W\\
-&\int_{0}^{t}(DB(\Phi(x,s)))^{2}(D \Phi)^{-1}(x,s)\d s,
\end{align*}
where $I$ is the identity matrix.

\begin{definition}[Invariant measure]\rm
A Borel probability measure $\mu$ on $\R^{d}$ is called an {\em invariant measure} or {\em stationary distribution} of Markov process $\left\{\Phi(x,t)\right\}_{t\geq 0}$ if  $P_{t}^{\ast}\mu=\mu$ for all $t\geq 0$, where $\left\{P_{t}^{\ast}\right\}_{t\geq0}$ is the Markov semigroup generated by $\left\{\Phi(x,t)\right\}_{t\geq 0}$.
\end{definition}

Next, we state the results about the relationship between the solution $\left\{\Phi(x,t)\right\}_{t\geq 0}$ and its invariant measure $\mu$. The results regarding invariant measures and ergodic measures are established separately by Ohno \cite{Ohno} and Kakutani \cite{Kakutani}, and can be found in Kifer \cite[p. 18]{Kifer}.
\begin{prop}
A probability measure $\mu$ is $P_{t}^{\ast}$-invariant (respectively, ergodic) if and only if $\mu\times\P$ is $F$-invariant (respectively, ergodic), where $F(x,\omega):=(\Phi(x,\omega),\theta\omega)$.
\end{prop}

In this paper, we assume that equation \eqref{equ.1.1} possesses an invariant measure, denoted by $\mu$. Additionally, denote the invariant measures corresponding to SDEs \eqref{equ.1.2} by $\mu_{k},k\in\N^{+}$. As is well-known, $\mu_{k}$ is said to converge to $\mu$ under the $\text{weak}^{\ast}$-topology if for any $\phi\in C_{b}(\R^{d})$,
$$
\lim_{k\rightarrow\infty}\int_{\R^{d}}\phi(x)\d\mu_{k}=\int_{\R^{d}}\phi(x)\d\mu,
$$
where $C_{b}(\R^{d})$ is the set of all continuous bounded functions defined on $\R^{d}$.

\begin{definition}[Measure-theoretic entropy of stochastic flows]\rm
Let \{$\Phi(t,\omega),t\geq0$\} be a one-parameter stochastic flow with an invariant measure $\mu$.
And $\xi$ is a finite measurable partition of $\R^{d}$. Denote the \textit{time-$t_{0}$-map measure-theoretic entropy of $\Phi$ with respect to $\xi$} as
\begin{align}\label{equ.2.2}
h^{t_{0}}_{\mu}(\Phi,\xi):=\lim\limits_{n\rightarrow\infty}\dfrac{1}{n}\int_{\Omega}H_{\mu}\Big(\bigvee_{i=0}^{n-1}\Phi^{-1}(it_{0},\omega)\xi\Big)\d\P(\omega).
\end{align}
Then 
$$
h_{\mu}^{t_{0}}(\Phi):=\sup_{\xi\in\mathcal{A}}h^{t_{0}}_{\mu}(\Phi,\xi),
$$
where the supremum is taken over the set of all finite measurable partitions of $\R^{d}$,
is called the \textit{time-$t_{0}$-map measure-theoretic entropy of $\Phi$}.
\end{definition}
\begin{remark}\rm\label{rem.2.6}
\begin{itemize}
\item [1.] By the Kingman subadditive ergodic theorem, we know that the limit \eqref{equ.2.2} exists and satisfies
\begin{align}
h^{t_{0}}_{\mu}(\Phi,\xi)=\inf_{n\geq1}\dfrac{1}{n}\int_{\Omega}H_{\mu}\Big(\bigvee_{i=0}^{n-1}\Phi^{-1}(it_{0},\omega)\xi\Big)\d\P(\omega).\notag
\end{align}
\item [2.] For any $t_{0}\geq0$, the entropy satisfies the relations
\begin{align}
h_{\mu}^{t_{0}}(\Phi)=h_{\mu}^{t_{0}}(\Phi^{-1}),\qquad h_{\mu}^{t_{0}}(\Phi)=t_{0}h_{\mu}^{1}(\Phi).\notag
\end{align}
\end{itemize}
\end{remark}

For solution $\left\{\Phi(x,t)\right\}_{t\geq 0}$ with the invariant measure $\mu$,  the multiplicative ergodic theorem, as stated by Liu and Qian \cite[p. 119]{Liupeidong2}, provides insights into its Lyapunov exponents:
\begin{prop}\label{lya-non}
If solution $\left\{\Phi(x,t,\omega)\right\}_{t\geq 0}$ has the following property
\begin{equation}\label{equ.2.3}
\int_{\R^{d}\times\Omega}\log^{+}\Vert D\Phi(x,1,\omega)\Vert+\log^{+}\Vert D\Phi^{-1}(x,1,\omega)\Vert\d\mu\times\P< +\infty,
\end{equation}
then there exists a subset $\Gamma\subset \R^{d}\times\Omega$ with $\mu\times\P(\Gamma)=1$ such that for any $(x,\omega)\in\Gamma$, there is an integer $p(x)$ and exists a sequence $\left\{ E_{i}(x,\omega)\right\}_{1\leq i\leq p(x)}$ of linear subspaces of $\R^{d}$:
$$
E(x,\omega)=	E_{1}(x,\omega)\oplus\cdots\oplus E_{p(x)}(x,\omega)\cong\R^{d}
$$
and numbers independent of $\omega\in\Omega$
$$
\lambda_{1}(x)>\lambda_{2}(x)>\cdots>\lambda_{p(x)}(x),
$$
such that for any $v\in E_{i}(x,\omega)$
$$	
\lim\limits_{t \rightarrow \infty}\dfrac{1}{t}\ln\vert D\Phi(x,t,\omega)v\vert=:\lambda_{i}(x),\quad 1\leq i\leq p(x).
$$
We call $E_{i}(x,\omega),1\leq i\leq p(x)$ the \textup{Oseledets subspaces} and the $m_{i}(x)=\textup{dim}E_{i}(x,\omega)$ the \textup{multiplicity} of the corresponding \textup{Lyapunov exponent} $\lambda_{i}(x)$.
\end{prop}

\begin{remark} \rm
\begin{enumerate}
\item[(1)] If coefficients $A,B$ of equation \eqref{equ.1.1} are $C^{1}$ functions and their derivatives are bounded, then it is immediate to see that the Jacobian matrix $D\Phi(x,t,\omega)$  for equation \eqref{equ.1.1} satisfies the condition \eqref{equ.2.3}; see Lemma \ref{lem.4.2}.
\item[(2)]	If $\mu$ is an ergodic measure, then $ \lambda_{i}(x), p(x),m_{i}(x)$  is independent of $x$.
\item[(3)]  The Lyapunov exponents obtained from continuous-time systems are equivalent to those obtained from discrete-time systems, as expressed by the equality:
$$
\lim\limits_{t \rightarrow \infty}\dfrac{1}{t}\ln\vert D\Phi(x,t,\omega)v\vert=\lim\limits_{n \rightarrow \infty}\dfrac{1}{n}\ln\vert D\Phi(x,n,\omega)v\vert,\quad t\in\R^{+},\quad n\in \N^{+}
$$
for any $v\in \R^{d}$ and almost all $\omega$.
\end{enumerate}
\end{remark}

\section{Continuity of measure-theoretic entropy}\label{sec.3}

As usual, we consider the time-$1$-map measure-theoretic entropy of solution $\{\Phi(x,t)\}_{t\geq 0}$, which is denoted by $h_{\mu}(\Phi):=h_{\mu}^{1}(\Phi)$.
So the measure-theoretic entropy $h_{\mu}(\Phi)$ can be expressed as
\begin{align}
h_{\mu}(\Phi)=\sup_{\xi}\lim\limits_{n\rightarrow\infty}\dfrac{1}{n}\int_{\Omega}H_{\mu}\Big(\bigvee_{i=0}^{n-1}\Phi^{-1}(i,\omega)\xi\Big)\d\P(\omega).\notag
\end{align}
Our objective is to determine under what conditions $h_{\mu}(\Phi)$ is upper semicontinuous with respect to coefficients of equations \eqref{equ.1.1} and \eqref{equ.1.2}.

First, we present some auxiliary results related to SDEs which will be used later. Since the results are elementary, we omit their proof.
\begin{lemma}\label{lem.4.2}
For equation \eqref{equ.1.1}, let coefficients $A,B$ belong to $ C^{1}(\R^{d})$, $\vert DA\vert,\vert DB\vert$ be bounded by a constant $K$.
Then for any $x,y\in\R^{d}$ and $t\in[0,T]$ with $T\in[0,+\infty)$, the solution $\Phi(x,t)$ satisfies
$$
E\vert \Phi(x,t)-\Phi(y,t)\vert^{2}\leq C\vert x-y\vert^{2},
$$
where $C=C(K,T)>0$ is a constant dependent only on $K,T$.
The Jacobian $D\Phi=(\partial_{1}\Phi,\partial_{2}\Phi,\dots,\partial_{d}\Phi)$ exists and satisfies  that for any $x\in\R^{d}$, $t\in [0,T]$ and $p \geq 2$
$$
E\vert D\Phi(x,t)\vert^{p}\leq C,\qquad E\vert D\Phi^{-1}(x,t)\vert^{p}\leq C.
$$
Here $C=C(p,K,T)>0$ is a constant. Furthermore, if $A,B$ are $C^{1,1}$ functions and their Lipschitz constants are bounded by $L$, then for any $t\in[0,T]$
\begin{align}
E\vert D\Phi(x,t)-D\Phi(y,t)\vert^{2}\leq C\vert x-y\vert^{2},\notag
\end{align}
where $C=C(K,L,T)>0$ is a constant.
\end{lemma}
\begin{lemma}\label{lem.3.1}
For equations \eqref{equ.1.1} and \eqref{equ.1.2}, assume that coefficients $A,B,A_{k},B_{k}$ satisfy the
global Lipschitz and global linear growth conditions, with the same Lipschitz constant $L$. If
\begin{align}
\lim\limits_{k\rightarrow\infty} \Vert A_{k}-A\Vert_{C^{0}}+\Vert B_{k}-B\Vert_{C^{0}}=0,\notag
\end{align}
then we have the continuous dependence of the solution on the coefficients: for any $T\in[0,+\infty)$
\begin{align}
\lim\limits_{k\rightarrow\infty}\sup_{x\in \R^{d}}\sup_{t\in[0,T]}E\Vert\Phi_{k}(x,t)-\Phi(x,t)\Vert^{2}=0.\notag
\end{align}
\end{lemma}

In the subsequent discussion, in order to emphasize the number of iterations of solution $\{\Phi(x,t)\}_{t\geq0}$, we will place the time index in the superscript, i.e. $\Phi^{n}(x,\omega)=\Phi(x,n,\omega)$ and $\Phi^{n}(\omega)(\cdot)=\Phi(\cdot,n,\omega)$. The same convention applies to $D\Phi$, similarly for the solution $\Phi_{k}$ of \eqref{equ.1.2}.

Let $\Lambda\subset\R^{d}$ be a compact subset. For $n\in\N^{+}$, $\epsilon>0$ and $x\in\Lambda$, set
\begin{align}
W^{s}(x,n,\epsilon,\omega):=\left\{y\in \R^{d}:d(\Phi^{i}(x,\omega),\Phi^{i}(y,\omega))\leq\epsilon, \text{ for } i\in[0,n-1]\right\}.\notag
\end{align}
Let $\text{card} A$ denote the cardinality of set $A$.
Define  $r(\Phi,n,\delta,\epsilon,\Lambda,x,\omega):=\text{card} F_{n}(x,\omega)$, where $F_{n}(x,\omega)\subset\Lambda\cap W^{s}(x,n,\epsilon,\omega)$ is the maximal set such that for any $y\neq z$ in $F_{n}(x,\omega)$, there exists some integer $i\in[0,n-1]$ such that $$
d(\Phi^{i}(x,\omega),\Phi^{i}(y,\omega))>\delta.
$$
Denote $r(\Phi,n,\delta,\epsilon,\Lambda,\omega):=\sup_{x\in\Lambda}r(\Phi,n,\delta,\epsilon,\Lambda,x,\omega)$ and
$$
r(\Phi,\epsilon,\Lambda):=\lim_{\delta\rightarrow0}\limsup\limits_{n\rightarrow\infty}\dfrac{1}{n} \int_{\Omega}\ln r(\Phi,n,\delta,\epsilon,\Lambda,\omega)\d\P,
$$
which captures the exponential growth rate of the maximal $(n,\delta)$-separated set within the local stable set.

\begin{lemma}\label{lem.3.3}
For equation \eqref{equ.1.1}, recalling that $\{\Phi(x,t)\}_{t\geq0}$ is its unique solution and $\mu$ is an invariant measure. Let $\Lambda\subset\R^{d}$ be a compact subset with $\mu(\Lambda)>\sigma$ for some constant $0<\sigma<1$.
For any $\epsilon>0$, if $\beta$ is a finite measurable partition of $\R^{d}$ with $\textup{diam}(\beta)<\epsilon$ on $\Lambda$, then
$$
h_{\mu}(\Phi)\leq h_{\mu}(\Phi,\beta)+r(\Phi,\epsilon,\Lambda).
$$
\end{lemma}
\begin{proof}
For the partition $\beta$ and sufficiently large $n\in\N^{+}$, denote
\begin{align}
\beta^{n}(\omega):=\bigvee_{i=0}^{n-1}\Phi^{-i}(\omega)(\beta).\notag
\end{align}
Let $k,N\in\N^{+}$ and $n:=kN$. Note that for any finite measurable partition  $\alpha=\{E_{1},\cdots,E_{r},\R^{d}\setminus\cup_{i=1}^{r} E_{i}\}$, we have that for almost all $\omega$
\begin{align}\label{equ.3.6}
H_{\mu}(\bigvee_{i=0}^{k-1}\Phi^{-iN}(,\omega)(\alpha))\leq H_{\mu}(\beta^{n}(\omega))+H_{\mu}(\bigvee_{i=0}^{k-1}\Phi^{-iN}(\omega)(\alpha)|\beta^{n}(\omega)).
\end{align}
Let $\gamma:=\left\{\Lambda,\R^{d}\setminus\Lambda\right\}$, $\beta_{1}^{n}(\omega):=\beta^{n}(\omega)\vee\gamma$. Let $\delta>0$ be smaller than $1/2$ of the minimum distance between any of the $E_{i}$'s. Let $B\in\beta_{1}^{n}(\omega)$ and $x_{B}\in B$. Let $F(x_{B},\omega)$ be a maximal subset of $\Lambda\bigcap W^{s}(x_{B},n,\epsilon,\omega)$ such that for $x\neq y\in F(x_{B},\omega)$ there exists $j\in[0,k)$ with $d(\Phi^{jN}(x,\omega),\Phi^{jN}(y,\omega))>\delta$. Clearly, $F(x_{B},\omega)$ is a $(n,\delta)$-separated, so
\begin{align}\label{equ.3.1}
\text{card}F(x_{B},\omega)\leq \text{card} F_{n}(x,\omega)= r(\Phi,n,\delta,\epsilon,\Lambda,\omega).
\end{align}
Notice that for the first term of \eqref{equ.3.6}
\begin{align}\label{equ.3.7}
&H_{\mu}(\bigvee_{i=0}^{k-1}\Phi^{-iN}(\omega)(\alpha)|\beta^{n}(\omega))\notag\\
\leq& H_{\mu}(\bigvee_{i=0}^{k-1}\Phi^{-iN}(\omega)(\alpha)\bigvee\gamma|\beta^{n}(\omega))\notag\\
=&H_{\mu}(\gamma|\beta^{n}(\omega))+H_{\mu}(\bigvee_{i=0}^{k-1}\Phi^{-iN}(\omega)(\alpha)|\beta_{1}^{n}(\omega))\notag\\
\leq &\ln 2+H_{\mu}(\bigvee_{i=0}^{k-1}\Phi^{-iN}(\omega)(\alpha)|\beta_{1}^{n}(\omega))
\end{align}
holds by the fact $H_{\mu}(\gamma|\beta^{n}(\omega))\leq H_{\mu}(\gamma)=\ln 2$. And
\begin{align}\label{equ.3.8}
&H_{\mu}(\bigvee_{i=0}^{k-1}\Phi^{-iN}(\omega)(\alpha)|\beta_{1}^{n}(\omega))\notag\\
=&\sum_{B\in\beta_{1}^{n}(\omega)}H_{\mu}(\bigvee_{i=0}^{k-1}\Phi^{-iN}(\omega)(\alpha)|B)\mu(B)  \notag\\
=&\sum_{\substack{B\in\beta_{1}^{n}(\omega)\\B\subset \Lambda}}H_{\mu}(\bigvee_{i=0}^{k-1}\Phi^{-iN}(\omega)(\alpha)|B)\mu(B)+\sum_{\substack{B\in\beta_{1}^{n}(\omega)\\B\cap\Lambda=\varnothing}}H_{\mu}(\bigvee_{i=0}^{k-1}\Phi^{-iN}(\omega)(\alpha)|B)\mu(B)\notag\\
\leq&\sum_{\substack{B\in\beta_{1}^{n}(\omega)\\B\subset \Lambda}}H_{\mu}(\bigvee_{i=0}^{k-1}\Phi^{-iN}(\omega)(\alpha)|B)\mu(B)+k\cdot \ln(\text{card}(\alpha))\cdot \mu(\R^{d}\setminus\Lambda).
\end{align}
For the first term above, notice that for each $B\in\beta_{1}^{n}(\omega)$ with $B\in\Lambda$, and each $A\in\bigvee_{i=0}^{k-1}\Phi^{-iN}(\omega)(\alpha)|B$, we can select a representative point $x_{A}\in A$ and assign to it an element $\phi(A)\in F(x_{B},\omega)$ such that $$d(\Phi^{jN}(x_{A},\omega),\Phi^{jN}(\phi(A),\omega))<\delta,\qquad\text{for all } j\in[0,k-1].$$
For a fixed $\phi(A)$, there are at most $2^{k}$ possible choices $A$ that can be assigned to it. Hence
\begin{align}\label{equ.3.9}
\sum_{\substack{B\in\beta_{1}^{n}\\B\subset \Lambda}}H_{\mu}(\bigvee_{i=0}^{k-1}\Phi^{-iN}(\omega)(\alpha)|B)\mu(B)\leq k\ln 2+\ln \text{card} F(x_{B},\omega).
\end{align}
Conclusionly, by \eqref{equ.3.7}, \eqref{equ.3.8}, \eqref{equ.3.9}
and the fact \eqref{equ.3.1}, we get
\begin{align}\label{equ.3.10}
&\dfrac{1}{kN}H_{\mu}(\bigvee_{i=0}^{k-1}\Phi^{-iN}(\omega)(\alpha)|\beta^{kN}(\omega))\notag\\
\leq& \dfrac{\ln 2}{kN}+\dfrac{\ln 2}{N}+\dfrac{1}{kN}\ln r(\Phi,kN,\delta,\epsilon,\Lambda,\omega)+\dfrac{1}{N}\ln \text{card}(\alpha)\cdot \mu(\R^{d}\setminus\Lambda).
\end{align}
Given $\epsilon_{1}>0$, let $N$ be a sufficiently large number such that $\frac{\ln 2}{N}<\epsilon_{1}$ and $\frac{1}{N}\ln\text{card}(\alpha)\mu(\R^{d}\setminus\Lambda)<\epsilon_{1}$. Then by \eqref{equ.3.6} and \eqref{equ.3.10}, we get
\begin{align}
&h_{\mu}(\Phi,\alpha)-h_{\mu}(\Phi,\beta)\notag\\
=&\dfrac{1}{N}h_{\mu}(\Phi^{N},\alpha)-h_{\mu}(\Phi,\beta)\notag\\
=&\limsup_{k\rightarrow\infty}\dfrac{1}{kN}\int_{\Omega}H_{\mu}(\bigvee_{i=0}^{k-1}\Phi^{-iN}(\omega)(\alpha))\d\P-\limsup_{k\rightarrow\infty}\dfrac{1}{kN}\int_{\Omega}H_{\mu}(\beta^{kN}(\omega))\d\P\notag\\
\leq& \limsup_{k\rightarrow\infty}\dfrac{1}{kN}\int_{\Omega} H_{\mu}(\bigvee_{i=0}^{k-1}\Phi^{-iN}(\omega)(\alpha)|\beta^{kN}(\omega))\d\P\notag\\
\leq& \dfrac{2}{N}+\dfrac{1}{N}\ln \text{card}(\alpha)\cdot \mu(\R^{d}\setminus\Lambda)+\limsup_{k\rightarrow\infty}\frac{1}{kN}\int_{\Omega}\ln r(\Phi,kN,\delta,\epsilon,\Lambda,\omega)\d\P\notag\\
\leq &2\epsilon_{1}+\limsup\limits_{n\rightarrow\infty}\dfrac{1}{n} \int_{\Omega}\ln r(\Phi,n,\delta,\epsilon,\Lambda,\omega)\d\P.\notag
\end{align}
Taking the limit $\delta\rightarrow0$ and the supremum over all finite measurable partitions, we get that
\begin{align}
h_{\mu}(\Phi)=\sup_{\alpha}h_{\mu}(\Phi,\alpha)\leq h_{\mu}(\Phi,\beta)+r(\Phi,\epsilon,\Lambda).\notag
\end{align}
We complete the proof.
\end{proof}

\begin{lemma}[Lyapunov norm]\label{lem.3.4}
For equation \eqref{equ.1.1},  let coefficients $A,B$ belong to $ C^{1}(\R^{d})$, $\vert DA\vert,\vert DB\vert$ be bounded by a constant $K$ and $\mu$ be an invariant measure.
Then there exists a subset $\Gamma\subset\R^{d}\times\Omega$ with the following properties:
\begin{itemize}
\item[(1)] $\mu\times\P(\Gamma)=1$.
\item[(2)] There exists a splitting $E(x,\omega)=E_{1}(x,\omega)\oplus E_{2}(x,\omega)\oplus E_{3}(x,\omega)$ depending measurably on $(x,\omega)\in \Gamma$.
\item[(3)] Suppose that $\lambda_{1}(x)$ is a negative measurable function and $\lambda_{3}(x),\epsilon(x)$ are positive measurable functions such that $v\in E_{1}(x,\omega)\Rightarrow \lambda(x,v)< \lambda_{1}(x)+\epsilon(x)<0$ and $v\in E_{3}(x,\omega)$ implies $\lambda(x,v)>\lambda_{3}(x)-\epsilon(x)>0$. Then there is a measurable norm $\vert\cdot\vert'_{x,\omega}$ on $E(x,\omega)$ for $(x,\omega)\in \Gamma$ and a measurable function $A(x,\omega)$ such that
\begin{itemize}
\item [(a)]$
\left\{
\begin{aligned}
&v\in E_{1}(x,\omega)\Rightarrow
\vert D\Phi(x,\omega)v\vert'_{F(x,\omega)}
\leq e^{\lambda_{1}(x)+\epsilon(x)}
\vert v\vert'_{x,\omega}, \\[6pt]
&v\in E_{2}(x,\omega)\Rightarrow
e^{-\epsilon(x)}\vert v\vert'_{x,\omega}
\leq\vert D\Phi(x,\omega)v\vert'_{F(x,\omega)}
\leq e^{\epsilon(x)}\vert v\vert'_{x,\omega}, \\[6pt]
&v\in E_{3}(x,\omega)\Rightarrow
\vert D\Phi(x,\omega)v\vert'_{F(x,\omega)}\geq e^{\lambda_{3}(x)-\epsilon(x)}\vert v\vert'_{x,\omega};
\end{aligned}
\right.
$
\item[(b)] $\vert v\vert\leq \vert v\vert'_{x,\omega}\leq A(x,\omega)\vert v\vert$ for $(x,\omega)\in\Gamma$ and $v\in E(x,\omega);$
\item[(c)] for any $(x,\omega)\in\Gamma$ $$\lim\limits_{n\rightarrow\pm\infty}\dfrac{1}{n}\ln A(F^{n}(x,\omega))=0.$$
\end{itemize}
\end{itemize}
\end{lemma}
\begin{proof}
Assume first that $\mu$ is ergodic. The extension to an arbitrary invariant measure is standard.
By the Oseledets multiplicative ergodic theorem (Proposition \ref{lya-non}), we know that there exist a subset $\Gamma\subset \R^{d}\times\Omega$ with $\mu\times\P(\Gamma)=1$ such that for any $(x,\omega)\in\Gamma$, there exist a splitting $E(x,\omega)=E_{1}(x,\omega)\oplus E_{2}(x,\omega)\oplus E_{3}(x,\omega)$ and constants $\epsilon>0,\lambda_{1}<0,\lambda_{3}>0$ satisfying
\begin{align}\label{equ.3.11}
&v\in E_{1}(x,\omega)\Rightarrow\lim\limits_{n\rightarrow\infty}\dfrac{1}{n}\ln\vert D\Phi^{n}(x,\omega)v\vert<\lambda_{1}+\epsilon, \notag\\
&v\in E_{2}(x,\omega)\Rightarrow -\epsilon<\lim\limits_{n\rightarrow\infty}\dfrac{1}{n}\ln\vert D\Phi^{n}(x,\omega)v\vert<\epsilon,\\
&v\in E_{3}(x,\omega)\Rightarrow \lim\limits_{n\rightarrow\infty}\dfrac{1}{n}\ln\vert D\Phi^{n}(x,\omega)v\vert>\lambda_{3}-\epsilon.\notag
\end{align}
Then for any $v\in E_{1}(x,\omega)$, there exists a measurable function $C_{1}(x,\omega)$ defined on $\Gamma$ such that
$$
\vert D\Phi^{n}(x,\omega)v\vert\leq C_{1}(x,\omega)e^{(\lambda_{1}+\epsilon) n}\vert v\vert.
$$
Let
$$
\vert v\vert'_{x,\omega,1}:=\sup_{n\geq 0}\vert D\Phi^{n}(x,\omega)v\vert e^{-(\lambda_{1}+\epsilon)n}
$$
and
$$
A_{1}(x,\omega):=\sup_{n\geq 0}\Vert D\Phi^{n}(x,\omega)\Vert e^{-(\lambda_{1}+\epsilon)n}.
$$
Notice that
\begin{align}\label{equ.3.14}
A_{1}(x,\omega)\leq C_{1}(x,\omega),\qquad\vert v\vert\leq \vert v\vert'_{x,\omega,1}\leq A_{1}(x,\omega)\vert v\vert.
\end{align}
Moreover,
\begin{align}\label{equ.3.12}
\vert D\Phi(x,\omega)v\vert'_{F(x,\omega),1}=&\sup_{n\geq 0}\vert D\Phi^{n}(F(x,\omega))D\Phi(x,\omega)v\vert  e^{-(\lambda_{1}+\epsilon)n}\notag\\
=& e^{\lambda_{1}+\epsilon}\Big(\sup_{n\geq 0}\vert D\Phi^{n+1}(x,\omega)\vert e^{-(\lambda_{1}+\epsilon)(n+1)}\Big)\notag\\
\leq& e^{\lambda_{1}+\epsilon}\vert v\vert'_{x.\omega,1}.
\end{align}
We know that
\begin{align}
A_{1}(F^{-1}(x,\omega))=&\sup_{n\geq 0}\Vert D\Phi^{n}(F^{-1}(x,\omega))\Vert e^{-(\lambda_{1}+\epsilon)n}\notag\\
\leq& \sup\big(1,\sup_{n\geq 1}\Vert D\Phi^{n-1}(x,\omega)\Vert\Vert D\Phi(F^{-1}(x,\omega))\Vert e^{-(\lambda_{1}+\epsilon)n}\big)\notag\\
\leq &\sup \big(1,e^{-(\lambda_{1}+\epsilon)}\Vert D\Phi(F^{-1}(x,\omega))\Vert A_{1}(x,\omega)\big)\notag\\
\leq & A_{1}(x,\omega)\sup \big(1,e^{-(\lambda_{1}+\epsilon)}\Vert D\Phi(F^{-1}(x,\omega))\Vert\big).\notag
\end{align}
So
$$
\ln A_{1}(F^{-n}(x,\omega))-\ln A_{1}(x,\omega)\leq \ln \Big(\sup\big(1,e^{-(\lambda_{1}+\epsilon)}\Vert D\Phi(F^{-1}(x,\omega))\Vert\big)\Big),
$$
whose right-hand side is integrable. Then by lemma in Ma$\tilde{\mathrm{n}}\acute{\mathrm{e}}$ \cite[p. 551]{Mane1981}, we get that
\begin{align}\label{equ.3.13}
\lim\limits_{n\rightarrow\pm\infty}\dfrac{1}{n}\ln A_{1}(F^{n}(x,\omega))=0.
\end{align}
By \eqref{equ.3.12}, \eqref{equ.3.14} and \eqref{equ.3.13}, it follows that the norm $\vert\cdot\vert_{x,\omega,1}$ and the function $A_{1}(x,\omega)$ satisfy the first property of $(a)$, as well as $(b)$ and $(c)$.

Next, by the second term of \eqref{equ.3.11}, we know that for any $v\in E_{2}(x,\omega)$ there exist measurable functions $C_{2}(x,\omega)$ and $C'_{2}(x,\omega)$ such that
$$
\vert D\Phi^{n}(x,\omega)v\vert\leq C_{2}(x,\omega)e^{\epsilon n}\vert v\vert,\qquad \vert D\Phi^{-n}(x,\omega)v\vert\leq C'_{2}(x,\omega)e^{\epsilon n}\vert v\vert.
$$
Denote
$$\vert v\vert'_{x,\omega,2}:=\max\left\{\sup_{n\geq0}\vert D\Phi^{n}(x,\omega)v\vert e^{-\epsilon n},\sup_{n<0}\vert D\Phi^{n}(F^{n}(x,\omega))v\vert e^{\epsilon n}\right\}$$
and
$$
A_{2}(x,\omega):=\max\left\{\sup_{n\geq0}\Vert D\Phi^{n}(x,\omega)\Vert e^{-\epsilon n},\sup_{n<0}\Vert D\Phi^{n}(F^{n}(x,\omega))\Vert e^{\epsilon n}\right\}.
$$
It is not difficult to find that
\begin{align}\label{equ.3.21}
A_{2}(x,\omega)\leq \max\left\{C_{2}(x,\omega),C'_{2}(x,\omega)\right\},\qquad \vert v\vert\leq \vert v\vert'_{x,\omega,2}\leq A_{2}(x,\omega)\vert v\vert.
\end{align}
Note that
\begin{align}\label{equ.3.18}
&\vert D\Phi(x,\omega)v\vert'_{F(x,\omega),2}\notag\\=&\max\left\{\sup_{n\geq0}\vert D\Phi^{n}(F(x,\omega))D\Phi(x,\omega)v\vert e^{-\epsilon n},\sup_{n<0}\vert D\Phi^{n}(F^{n+1}(x,\omega))D\Phi(x,\omega)v\vert e^{\epsilon n}\right\}\notag\\
\leq&\max\left\{\sup_{n\geq0}\vert D\Phi^{n+1}(x,\omega)v\vert e^{-\epsilon n},\sup_{n<0}\vert D\Phi^{n}(F^{n}(x,\omega))v\vert e^{\epsilon (n)}\right\}\notag\\
\leq& e^{\epsilon}\vert v\vert'_{x,\omega,2}
\end{align}
and
\begin{align}\label{equ.3.19}
&\vert D\Phi^{-1}(F^{-1}(x,\omega))v\vert'_{F^{-1}(x,\omega),2}\notag\\
=&\max\{\sup_{n\geq0}\vert D\Phi^{n}(F^{-1}(x,\omega))D\Phi^{-1}(F^{-1}(x,\omega))v\vert e^{-\epsilon n},\notag\\
&\qquad\qquad\sup_{n<0}\vert D\Phi^{n}(F^{n-1}(x,\omega))D\Phi^{-1}(F^{-1}(x,\omega))v\vert e^{\epsilon n}\}\notag\\
\leq&\max\left\{\sup_{n\geq0}\vert D\Phi^{n}(x,\omega)v\vert e^{-\epsilon (n-1)},\sup_{n<0}\vert D\Phi^{n}(F^{n}(x,\omega))v\vert e^{\epsilon (n+1)}\right\}\notag\\
\leq&e^{\epsilon}\vert v\vert'_{x,\omega,2}.
\end{align}
Moreover, we know that
\begin{align}
A_{2}(F(x,\omega))=&\max\left\{\sup_{n\geq0}\Vert D\Phi^{n}(F(x,\omega))\Vert e^{-\epsilon n},\sup_{n<0}\Vert D\Phi^{n}(F^{n+1}(x,\omega))\Vert e^{\epsilon n}\right\}\notag\\
=&\max\Big\{\sup_{n\geq0}\Vert D\Phi^{n}(F(x,\omega))\Vert\Vert D\Phi(x,\omega)\Vert e^{-\epsilon (n-1)},\notag\\
&\qquad\qquad\sup_{n<0}\Vert D\Phi^{n}(F^{n+1}(x,\omega))\Vert\Vert D\Phi(x,\omega)\Vert e^{\epsilon (n+1)}\Big\}\times\max\left\{1,\Vert D\Phi^{-1}(x,\omega)\Vert e^{-\epsilon}\right\}\notag\\
\leq&  A_{2}(x,\omega)\times\max\left\{1,\Vert D\Phi^{-1}(x,\omega)\Vert e^{-\epsilon}\right\},\notag
\end{align}
which implies
$$
\ln A_{2}(F(x,\omega))-\ln A_{2}(x,\omega)\leq \ln \max\left\{1,\Vert D\Phi^{-1}(x,\omega)\Vert e^{\epsilon}\right\}.
$$
Due to Ma$\tilde{\mathrm{n}}\acute{\mathrm{e}}$ \cite[p. 551]{Mane1981}, we have
\begin{align}\label{equ.3.20}
\lim\limits_{n\rightarrow\pm\infty}\dfrac{1}{n}\ln A_{2}(F^{n}(x,\omega))=0.
\end{align}
By \eqref{equ.3.18}, \eqref{equ.3.19}, \eqref{equ.3.21} and \eqref{equ.3.20}, we conclude that the norm $\vert\cdot\vert_{x,\omega,2}$ and the function $A_{2}(x,\omega)$ satisfy the second property of $(a)$, as well as $(b)$ and $(c)$.

The third term of \eqref{equ.3.11} implies that for any $v\in E_{3}(x,\omega)$
\begin{align}
\lim\limits_{n\rightarrow\infty}\dfrac{1}{n}\ln\vert D\Phi^{-n}(x,\omega)v\vert\leq -\lambda_{3}+\epsilon.\notag
\end{align}
We define
$$\vert v\vert'_{x,\omega,3}:=\sup_{n\geq 0}\vert D\Phi^{-n}(x,\omega)v\vert e^{(\lambda_{3}-\epsilon)n}$$
and
$$
A_{3}(x,\omega):=\sup_{n\geq 0}\Vert D\Phi^{-n}(x,\omega)\Vert e^{(\lambda_{3}-\epsilon)n}.
$$
It is straightforward to verify that $\vert\cdot\vert_{x,\omega,3}$ and $A_{3}(x,\omega)$ satisfy the third estimate in condition~$(a)$, as well as conditions $(b)$ and $(c)$.

To conclude, for $v=(v_{1},v_{2},v_{3})$ with $v_{1}\in E_{1}(x,\omega)$, $v_{2}\in E_{2}(x,\omega)$ and $v_{3}\in E_{3}(x,\omega)$, we define $$
\vert v\vert'_{x,\omega}:=\max\left\{\vert v\vert'_{x,\omega,1},\vert v\vert'_{x,\omega,2},\vert v\vert'_{x,\omega,3}\right\},\qquad A(x,\omega):=\max\left\{A_{1}(x,\omega),A_{2}(x,\omega),A_{3}(x,\omega)\right\}.
$$
The proof is complete.
\end{proof}

For any integer $1\leq k\leq d$, let $D^{k}$ be $[0,1]^{k}\subset\R^{k}$.
A smooth $k$-disc in $\R^{d}$ is a $C^{1,1}$ map $\gamma:D^{k}\rightarrow\R^{d}$.
We define the $k$-volume of $\gamma$ to be
\begin{align}
\vert\gamma\vert:=\int_{D^{k}}\vert\wedge^{k}T_{x}\gamma\vert\d x,\notag
\end{align}
where $T_{x}\gamma$ is the derivative of $\gamma$ at $x$, $\wedge^{k}T_{x}\gamma:\wedge^{k}T_{x}D^{k}\rightarrow\wedge^{k}\R^{d}$ is the induced linear map on the $d$-exterior power. Similarly, for a measurable subset $A\subset D^{k}$, we define the restricted volume
\begin{align}
\vert\gamma|A\vert:=\int_{A}\vert\wedge^{k}T_{x}\gamma\vert\d x.\notag
\end{align}
Let $\mathcal{D}$ be a collection of $C^{1,1}$-discs in $\R^{d}$ whose dimensions vary from $1$ through
$d$.
For a $C^{1,1}$-disc $\gamma$, we define the quantity
\begin{align}
V_{\text{loc}}(\gamma,n,\epsilon,\Phi(\omega)):=\sup_{x\in \R^{d}}\vert \Phi^{n}(\omega)\circ\gamma\big|\gamma^{-1}(W^{s}(x,n,\epsilon,\omega))\vert\notag
\end{align}
and the local growth rate of stable submanifolds under iteration
\begin{align}
\bar{G}_{\text{loc}}(\Phi,\epsilon):=\sup_{\gamma\in\mathcal{D}}\limsup_{n\rightarrow\infty}\dfrac{1}{n}\int_{\Omega}\ln^{+}V_{\text{loc}}(\gamma,n,\epsilon,\Phi(\omega))\d\P(\omega),\notag
\end{align}
where $\ln^{+}x=\max\left\{1,x\right\}$. Observe that $\gamma^{-1}(W^{s}(x,n,\epsilon,\omega))$ is the set of points $y$ in $D^{k}$ such that $d(\Phi^{j}(\omega)\circ\gamma(x),\Phi^{j}(\omega)\circ\gamma(y))<\epsilon$ for $0\leq j\leq n-1$ and $\vert \Phi^{n}(\omega)\circ\gamma\big|\gamma^{-1}(W^{s}(x,n,\epsilon,\omega))\vert$ is the  $k$-dimensional volume of the $\Phi^{n}(\omega)\circ\gamma$-image of this
set.

For deterministic dynamical systems, Newhouse \cite{Newhouse-1988} established a relationship between entropy and the growth rate of the volumes of iterated submanifolds. Following his approach, we derive an analogous relation for SDEs, linking local entropy to volume growth of local stable manifolds.

\begin{theorem}\label{lem.3.5}
For equation \eqref{equ.1.1},  let coefficients $A,B$ belong to $ C^{1,1}(\R^{d})$, $\vert DA\vert,\vert DB\vert$ be bounded by a constant $K$ and $\mu$ be an invariant measure. Then for any $\epsilon>0$ and $0<\sigma<1$, given a compact subset $\Lambda\subset\R^{d}$ with $\mu(\Lambda)>\sigma$, we have
$$
r(\Phi,\epsilon,\Lambda)\leq \bar{G}_{\textup{loc}}(\Phi,\epsilon).
$$
\end{theorem}
\begin{proof}
When $r(\Phi,\epsilon,\Lambda)=0$, the inequality trivially holds. If $r(\Phi,\epsilon,\Lambda)>0$,
then there exists at least one positive Lyapunov exponent by the Ruelle inequality, which implies $E_{3}(x,\omega)\neq \left\{0\right\}$. By Lemma \ref{lem.3.4}, note that there exists a full-measure subset $\Gamma\subset\R^{d}\times\Omega$ such that for each $(x,\omega)\in\Gamma$ there are a measurable norm $\vert\cdot\vert'_{x,\omega}$ on $E(x,\omega)$ and a measurable function $A(x,\omega)$ satisfying the stated properties.

By Lemma \ref{lem.4.2}, we know that there exists a measurable function $B:\Omega\rightarrow \R^{+}$ such that for any $x,y\in\R^{d}$
\begin{equation}\label{equ.3.15}
\begin{aligned}
\Vert D\Phi(y,\omega)-D\Phi(x,\omega)\Vert\leq B(\omega)\vert x-y\vert\quad \P-a.s.
\end{aligned}
\end{equation}
We now work in local coordinates. For a fixed $x$, if $y\neq x\in\R^{d}$, we consider $E_{i}(x,\omega),1\leq i\leq 3$ as a subspace of $E(y,\omega)$ by translation, and we denote $\vert v\vert'_{y,\omega}=\vert v\vert'_{x,\omega}$ for any $v\in E(y,\omega)$. Then for any $v\in E_{1}(x,\omega)$, we have
\begin{align}\label{equ.3.16}
\vert D\Phi(y,\omega)v\vert'_{F(y,\omega)}=&\vert D\Phi(y,\omega)v\vert'_{F(x,\omega)}\notag\\
=&\vert D\Phi(x,\omega)v+ D\Phi(y,\omega)v-D\Phi(x,\omega)v\vert'_{F(x,\omega)}\notag\\
\leq &e^{\lambda_{1}+\epsilon}\vert v\vert'_{x,\omega}+A(F(x,\omega))\Vert D\Phi(y,\omega)-D\Phi(x,\omega)\Vert\cdot\vert v\vert'_{x,\omega} \notag\\
\leq& e^{\lambda_{1}+\epsilon}\vert v\vert'_{x,\omega}+A(F(x,\omega))B(\omega)\vert y-x\vert\cdot\vert v\vert'_{x,\omega}
\end{align}
by the definition of $\vert\cdot\vert'_{x,\omega}$ and \eqref{equ.3.15}. If we denote
$$
\bar{\epsilon}(x,\omega):=\min\Big\{1,\dfrac{e^{\lambda_{1}+2\epsilon}-e^{\lambda_{1}+\epsilon}}{A(F(x,\omega))B(\omega)}\Big\},
$$
then by \eqref{equ.3.15} and \eqref{equ.3.16}, for any $\vert y-x\vert\leq\bar{\epsilon}(x,\omega)$ and $v\in E_{1}(x,\omega)$, we have
\begin{align}
\vert D\Phi(y,\omega)v\vert'_{F(y,\omega)}\leq e^{\lambda_{1}+2\epsilon}\vert v\vert'_{y,\omega}.\notag
\end{align}
Note that
$$
\lim\limits_{n\rightarrow\infty}\dfrac{1}{n}\ln\bar{\epsilon} (F^{n}(x,\omega))=0\quad \mu\times\P-a.s.
$$
For each $(x,\omega)\in\Gamma$, let $b(x,\omega):=\inf_{n\geq 0}\bar{\epsilon} (F^{n}(x,\omega))e^{\epsilon n}$. Then we have $0<b(x,\omega)\leq \bar{\epsilon} (x,\omega)$ and $b(F(x,\omega))\geq e^{-\epsilon}b(x,\omega)$. Since $\vert\ln b(F(x,\omega)-\ln b(x,\omega)\vert$ is integrable, then it follows from the lemma in Ma$\tilde{\mathrm{n}}\acute{\mathrm{e}}$ \cite[p. 551]{Mane1981} that there exists a full-measure subset $\Gamma\subset \R^{d}\times\Omega$ on which
$$
\lim\limits_{n\rightarrow\pm\infty}\dfrac{1}{n}\ln b(F^{n}(x,\omega))=0\quad \mu\times\P-a.s.
$$
So there exists a measurable function $C(x,\omega)$ such that $b(F^{n}(x,\omega))\geq C(x,\omega)e^{\epsilon n}$ for $n\leq 0$. Letting $\epsilon_{1}(x,\omega):=\inf_{n\leq 0}b(F^{n}(x,\omega))e^{-\epsilon n}$, we can get $0<\epsilon_{1}(x,\omega)\leq b(x,\omega)$ and for any $k\in \Z$
\begin{align}\label{equ.3.26}
\epsilon_{1}(F^{k}(x,\omega))\geq e^{-\vert k\vert\epsilon}\epsilon_{1}(x,\omega).
\end{align}
Reducing $\epsilon_{1}(x,\omega)$ if necessary, we assume that for any $v\in E_{2}(x,\omega)$
\begin{align}\label{equ.3.22}
e^{-2\epsilon}\vert v\vert'_{y,\omega}\leq \vert D\Phi(y,\omega)v\vert'_{F(y,\omega)}\leq e^{2\epsilon}\vert v\vert'_{y,\omega}.
\end{align}
and for any $v\in E_{3}(x,\omega)$
\begin{align}\label{equ.3.23}
\vert D\Phi(y,\omega)v\vert'_{F(y,\omega)}\geq e^{\lambda_{3}-2\epsilon}\vert v\vert'_{y,\omega}.
\end{align}

After defining the same norms in local coordinates, we can further define diameters of sets and volumes of submanifolds in  $ B_{\epsilon_{1}(x,\omega)}(x)$. Thus $\text{diam}'(A)$ denotes the diameter of a set $A$ induced by the metric on some $ B_{\epsilon_{1}(x,\omega)}(x)$ induced by $\vert\cdot\vert'_{x,\omega}$. Since the angle function $\eta(x,\omega)$ has subexponential growth
along orbits in $\Gamma$, we may assume that there is a measurable function $B:\Gamma\rightarrow\R^{+}$ such that for any smooth $d$-disc $\gamma$ in $ B_{\epsilon_{1}(x,\omega)}(x)$
\begin{align}\label{equ.3.27}
\vert \gamma\vert\geq B(x,\omega)\vert \gamma\vert'\geq B(x,\omega)\vert \gamma\vert.
\end{align}

Next, for the compact subset $\Lambda\subset\R^{d}$, the bundles $E_{1}(x,\omega)$, $E_{2}(x,\omega)$ and $E_{3}(x,\omega)$ and the functions $\epsilon_{1}(x,\omega)$ and $B(x,\omega)$ are continuous on $\Lambda$ for almost all $\omega\in\Omega$. Let $0<M_{2}(\omega)<\frac{1}{2}\min\{\inf_{x\in\Lambda}\left\{\epsilon_{1}(x,\omega)\right\},\inf_{x\in\Lambda}\left\{B(x,\omega)\right\}\}$ be such that
\begin{align}\label{equ.3.2}
\epsilon_{1}(F^{n}(x,\omega))\geq M_{2}(\omega)e^{-\epsilon n},\qquad  B(F^{n}(x,\omega))\geq M_{2}(\omega)e^{-\epsilon n}.
\end{align}

Let $\text{exp}_{x}$ denote the exponential map associated to the Riemannian metric. Let $D_{i}(x,\omega):=\text{exp}_{x}(B_{\epsilon_{1}(x,\omega)}(0)\cap E_{i}(x,\omega)),1\leq i\leq 3$ and $D(x,\omega):=\text{exp}_{x}(B_{\epsilon_{1}(x,\omega)}(0))$. For $\sigma_{i}\in(0,1),$ $i=1,2,3,$ let
$$
\sigma_{i}E_{i}(x,\omega):=\left\{v\in E^{i}(x,\omega):\vert v\vert'_{x,\omega}\leq\sigma_{i}\right\}
$$
and
$$
\sigma_{1}D_{1}(x,\omega)\times \sigma_{2}D_{2}(x,\omega)\times \sigma_{3}D_{3}(x,\omega):=\text{exp}_{x}(\sigma_{1}E_{1}(x,\omega)\times \sigma_{2}E_{2}(x,\omega)\times \sigma_{3}E_{3}(x,\omega)).
$$
Let $\pi_{i}:D(x,\omega)\rightarrow D_{i}(x,\omega)$ be the projection for $i=1,2,3$. Denote
$$
\text{Image }g:=\left\{(g(u,v),u,v):u\in\sigma_{2}D_{2}(x,\omega),v\in\sigma_{3}D_{3}(x,\omega)\right\},
$$
where $g:\sigma_{2}D_{2}(x,\omega)\times \sigma_{3}D_{3}(x,\omega)\rightarrow\sigma_{1}D_{1}(x,\omega)$ is a $C^{1,1}$ function.

It is stated in Liu and Qian \cite[p. 64]{Liupeidong2} that if for any $(x,\omega)\in\Gamma$, define
\begin{align}
W_{\text{loc}}^{s}(x,\omega):=\left\{y\in B_{\epsilon_{1}(x,\omega)}(x):d(\Phi^{n}(x,\omega),\Phi^{n}(y,\omega))\leq\epsilon_{1}(x,\omega)e^{(\lambda_{1}+\epsilon)n}, \text{ for } n\geq 0\right\},\notag
\end{align}
then $W_{\text{loc}}^{s}(x,\omega)$ is a $C^{1,1}$-disc tangent at $x$ in $E_{1}(x,\omega)$.
For each $x\in\Lambda$ choose a $C^{1}$-disc $W(x,\omega)$ tangent at $x$ to $E_{2}(x,\omega)\oplus E_{3}(x,\omega)$. For almost every $\omega\in\Omega$,  assume $M_{2}(\omega)$ small enough that if $x\in\Lambda$ and $y\in B_{M_{2}(\omega)}(x)\cap\Lambda$, then $W_{\text{loc}}^{s}(x,\omega)\cap W(y,\omega)\neq\emptyset$.

Note that
$$
r(\Phi,\epsilon,\Lambda)=\lim_{\delta\rightarrow 0}\limsup_{n\rightarrow\infty}\dfrac{1}{n}\int_{\Omega}\ln \sup_{x\in\Lambda}\text{card} (F_{n}(x,\omega))\d\P(\omega),
$$
where $F_{n}(x,\omega)$ is a maximal $(n,\delta)$-$\text{separated set in }\Lambda\cap W^{s}(x,n,\epsilon,\omega)$.
Then there exists some constant $\delta>0$ such that
$$
r(\Phi,\epsilon,\Lambda)\leq \epsilon_{0}+\limsup_{n\rightarrow\infty}\dfrac{1}{n}\int_{\Omega}\ln\sup_{x\in\Lambda} \text{card} (F_{n}(x,\omega))\d\P(\omega).
$$
Let $x_{1},x_{2},\cdots,x_{s}$ be finitely many points in $\Lambda$ such that $\Lambda\subset \cup_{i=1}^{s}B_{M_{2}(\omega)}(x_{i})$. Then there is some $x_{j}$ such that
\begin{align}
r(\Phi,\epsilon,\Lambda)\leq \epsilon_{0}+\limsup_{n\rightarrow\infty}\dfrac{1}{n}\int_{\Omega}\ln\sup_{x\in\Lambda} \text{card} (F_{n}(x,\omega)\cap B_{M_{2}(\omega)}(x_{j}))\d\P(\omega).\notag
\end{align}
Now choose a disc $\gamma$ such that $W=\text{Image}\gamma$ is $C^{1}$ near enough $W(x_{j},\omega)$ so that for any $y\in B_{M_{2}(\omega)}(x_{j})\cap \Lambda$, there exists a unique point $z(y)\in W^{s}_{\text{loc}}(y,\omega)\cap W$.

We choose $\sigma_{2},\sigma_{3}>0$ such that $W$ contains a set $W_{1}$ which contains $y$ and is the graph of a $C^{1,1}$-function $g:\sigma_{2}D_{2}(y,\omega)\times\sigma_{3}D_{3}(y,\omega)\rightarrow D_{1}(y,\omega)$ with $g(0,0)=z(y)$.
Let
\begin{align}\label{equ.3.29}
W_{n}(y,\omega):=\text{Image } g\big|e^{-4\epsilon(n-1)}\frac{\delta}{4}\epsilon_{1}(F^{n}(y,\omega))D_{2}(y,\omega)\times D_{3}(y,\omega)
\end{align}
and
\begin{align}\label{equ.3.28}
\tilde{B}_{n}(y,\omega):=\cap_{0\leq k\leq n-1} \Phi^{-k}(\omega)B_{\epsilon_{1}(F^{k}(y,\omega))}(\Phi^{k}(y,\omega)).
\end{align}
Let $\tilde{W}_{n}(y,\omega):=W_{n}(y,\omega)\cap \tilde{B}_{n}(y,\omega)$. Then, by \eqref{equ.3.22} and \eqref{equ.3.29}, we obtain
\begin{align}\label{equ.3.24}
\text{diam}'(\pi_{2}\Phi^{k}(\omega)\tilde{W}_{n}(y,\omega))\leq e^{2k\epsilon}\epsilon_{1}(F^{n}(y,\omega))e^{-4\epsilon(n-1)}\frac{\delta}{2}.
\end{align}
Moreover, by \eqref{equ.3.28}, we know that for any $0\leq k\leq n-1$
\begin{align}
\Phi^{k}(\omega)(\tilde{W}_{n}(y,\omega))\subset \Phi^{-(n-k-1)}(\theta^{n-1}\omega)( B_{\epsilon_{1}(F^{n-1}(y,\omega))}(\Phi^{n-1}(y,\omega))),\notag
\end{align}
which together with \eqref{equ.3.23} implies that
\begin{align}\label{equ.3.25}
\text{diam}'(\pi_{3}\Phi^{k}(\omega)\tilde{W}_{n}(y,\omega))\leq 2e^{-(\lambda_{3}-2\epsilon)(n-1-k)}\epsilon_{1}(F^{n}(y,\omega)).
\end{align}
Combining \eqref{equ.3.24} and \eqref{equ.3.25} yields
\begin{align}
\text{diam}(\Phi^{k}(\omega)\tilde{W}_{n}(y,\omega))\leq& \text{diam}'(\Phi^{k}(\omega)\tilde{W}_{n}(y,\omega))\notag\\
\leq &\epsilon_{1}(F^{n}(y,\omega))\max\left\{e^{2k\epsilon}e^{-4\epsilon(n-1)}\frac{\delta}{2},2e^{-(\lambda_{3}-2\epsilon)(n-1-k)}\right\}.\notag
\end{align}
Provided $n$ is large enough, one can find a constant $k_{1}\in\N^{+}$ with the property that for any $k\in [0,n-1-k_{1}]$, $\text{diam}(\Phi^{k}(\omega)\tilde{W}_{n}(y,\omega))\leq\frac{\delta}{2}$. This implies that whenever there exists some $k\in[0,n-1-k_{1}]$ such that $d(\Phi^{k}(x,\omega),\Phi^{k}(y,\omega))>\delta$, then $\tilde{W}_{n}(x,\omega)\cap\tilde{W}_{n}(y,\omega)=\varnothing$.

We select a subset
$$
F'_{n}(x,\omega)\subset F_{n}(x,\omega)\cap B_{M_{2}(\omega)}(x_{j}),
$$
which is the maximal set such that if $y\neq z\in F'_{n}(x,\omega)$, then $d(\Phi^{k}(x,\omega),\Phi^{k}(y,\omega))>\delta$ for some $k\in[0,n-1-k_{1}]$.
For any $y\in F'_{n}(x,\omega)$, it follows from \eqref{equ.3.26}, \eqref{equ.3.24} and \eqref{equ.3.25} that
\begin{align}
\vert\Phi^{n}(\omega)\tilde{W}_{n}(y,\omega)\vert'\geq& (e^{-2\epsilon(n-1)}\frac{\delta}{2}\epsilon_{1}(F^{n}(y,\omega))^{\text{dim} D_{2}} (\epsilon_{1}(F^{n}(y,\omega))^{\text{dim}D_{3}}\notag\\
=& (\frac{\delta}{2})^{\text{dim}D_{2}}(e^{-2\epsilon(n-1)})^{\text{dim} D_{2}}(\epsilon_{1}(F^{n}(y,\omega)))^{\text{dim}D_{2}+\text{dim}D_{3}}.\notag
\end{align}
Moreover, by \eqref{equ.3.27} and \eqref{equ.3.2}, we obtain
\begin{align}
\vert\Phi^{n}(\omega)\tilde{W}_{n}(y,\omega)\vert\geq& B(F^{n}(y,\omega))\vert\Phi^{n}(\omega)\tilde{W}_{n}(y,\omega)\vert'\notag\\
\geq &(\frac{\delta}{2})^{d} M_{2}(\omega)e^{-\epsilon(n-1)}e^{-2\epsilon(n-1)d}(M_{2}(\omega))^{d}e^{-\epsilon(n-1)d}\notag\\
\geq&(\frac{\delta}{2})^{d}(M_{2}(\omega))^{d+1} e^{-4\epsilon(n-1)d},\notag
\end{align}
which implies
$$
\vert \Phi^{n}(\omega)(W\cap\Lambda\cap W^{s}(n,\epsilon,x,\omega))\vert\geq (\text{card}F'_{n}(x,\omega))C_{1}(\omega) e^{-4\epsilon(n-1)d},
$$
where $C_{1}(\omega)=(\frac{\delta}{2})^{d}(M_{2}(\omega))^{d+1}$ is a measurable function.
Notice that
\begin{align}
&\limsup\limits_{n\rightarrow\infty}\frac{1}{n}\int_{\Omega}\ln^{+}\sup_{x\in\R^{d}}\vert \Phi^{n}(\omega)(W\cap\Lambda\cap W^{s}(n,\epsilon,x,\omega))\d\vert\P(\omega)\notag\\
\geq& \limsup\limits_{n\rightarrow\infty}\frac{1}{n}\int_{\Omega}\sup_{x\in\Lambda}\ln^{+}(\text{card}F'_{n}(x,\omega))C_{1}(\omega)e^{-4\epsilon(n-1)d}\d\P(\omega)\notag\\
=&\limsup\limits_{n\rightarrow\infty}\frac{1}{n}\int_{\Omega}\sup_{x\in\Lambda}\ln^{+}(\text{card}F_{n}(x,\omega))C_{1}(\omega)e^{-4\epsilon(n-1)d}\d\P(\omega),\notag
\end{align}
which leads to
$$
r(\Phi,\epsilon,\Lambda)\leq\bar{G}_{\text{loc}}(\Phi,\epsilon).
$$
The proof is complete.
\end{proof}

By employing the result of Kifer and Yomdin \cite{Kifer-Yomdin}, we can obtain an upper bound for the volume growth along local stable submanifolds.
\begin{lemma}\label{lem.3.6}
For equation \eqref{equ.1.1},  let coefficients $A,B$ belong to $ C^{r}(\R^{d})$ for some integer $1\leq r<\infty$, and their derivatives up to $r$-th order are bounded. If
$$
\int_{\Omega}\ln^{+}\Vert D\Phi(\omega)\Vert_{C^{r-1}}\d\P(\omega)<+\infty,
$$
then for any $\epsilon>0$
$$
\bar{G}_{\textup{loc}}(\Phi,\epsilon)\leq \dfrac{d}{r}R(\Phi),
$$
where
$$
R(\Phi):=\limsup\limits_{n\rightarrow\infty}\dfrac{1}{n}\ln^{+}\sup_{x}\Vert D\Phi^{n}(x,\omega)\Vert\leq\int_{\Omega}\ln^{+}\sup_{x}\Vert D\Phi(x,\omega)\Vert\d\P(\omega).
$$
In particular, for $A,B\in C^{\infty}(\R^{d})$ with bounded derivatives of all orders, we have
$$
\bar{G}_{\textup{loc}}(\Phi,\epsilon)=0.
$$
\end{lemma}

\begin{theorem}\label{the.3.1}
For equations \eqref{equ.1.1} and \eqref{equ.1.2}, let $A,B,A_{k},B_{k}$ be $C^{\infty}$ functions with bounded derivatives of all orders. Suppose that $\mu_{k}$ converges to $\mu$ under the $\text{weak}^{\ast}$-topology.
If for every $k\in \N^{+}$ and any integer $1\leq r<\infty$,
\begin{align}\label{pro.3.1}
\int_{\Omega}\ln^{+}\Vert D\Phi_{k}(\omega)\Vert_{C^{r-1}}\d\P(\omega)<+\infty
\end{align}
and
$$
\lim\limits_{k\rightarrow\infty}\Vert A_{k}-A\Vert_{C^{0}}+\Vert B_{k}-B\Vert_{C^{0}}=0,
$$
then the measure-theoretic entropy is upper semicontinuous:
\begin{equation}
\limsup\limits_{k\rightarrow\infty}h_{\mu_{k}}(\Phi_{k})\leq h_{\mu}(\Phi).\notag
\end{equation}
\end{theorem}
\begin{proof}
According to Lemma \ref{lem.3.3}, we know that for any $k>0$
\begin{align}\label{equ.3.17}
h_{\mu_{k}}(\Phi_{k})\leq& h_{\mu_{k}}(\Phi_{k},\beta)+r(\Phi_{k},\epsilon,\Lambda).
\end{align}
For the first term above, by Remark \ref{rem.2.6}, we have
\begin{align}
h_{\mu_{k}}(\Phi_{k},\beta)=\inf_{n\geq1}\dfrac{1}{n}\int_{\Omega}H_{\mu_{k}}(\bigvee_{i=0}^{n-1}\Phi_{k}^{i}(\omega)\beta)\d\P(\omega).\notag
\end{align}
We consider the continuity of $\int_{\Omega}H_{\mu_{k}}(\bigvee_{i=0}^{n-1}\Phi_{k}^{i}(\omega)\beta)\d\P(\omega)$ with respect to $k$ for a fixed $n$. By Lemma \ref{lem.3.1}, we know that for any $x\in\R^{d}$ and $t\in[0,T]$
\begin{align}
\lim\limits_{k\rightarrow\infty}\Vert \Phi_{k}(x,t,\omega)-\Phi(x,t,\omega)\Vert=0,\qquad \P-a.s.\notag
\end{align}
Suppose that $\mu(\partial C)=0$ for $C\in\beta$. Together with the property that $\mu_{k}$ converges to $\mu$ under the $\text{weak}^{\ast}$-topology, we have, for any $n\in\N^{+}$
\begin{align}
\lim\limits_{k\rightarrow\infty}\vert H_{\mu_{k}}\Big(\bigvee_{i=0}^{n-1}\Phi_{k}^{i}(\omega)\beta\Big)-H_{\mu}\Big(\bigvee_{i=0}^{n-1}\Phi^{i}(\omega)\beta\Big)\vert=0,\qquad \P-a.s.,\notag
\end{align}
which implies by the dominated convergence theorem that
\begin{align}
\lim\limits_{k\rightarrow\infty}\int_{\Omega}\vert H_{\mu_{k}}\Big(\bigvee_{i=0}^{n-1}\Phi_{k}^{i}(\omega)\beta\Big)-H_{\mu}\Big(\bigvee_{i=0}^{n-1}\Phi^{i}(\omega)\beta\Big)\vert\d\P(\omega)=0.\notag
\end{align}
As we know, the infimum of continuous functions is upper semicontinuous. Then the entropy with partition $\beta$ is upper semicontinuous
\begin{align}
\limsup\limits_{k\rightarrow\infty}h_{\mu_{k}}(\Phi_{k},\beta)\leq h_{\mu}(\Phi,\beta).\notag
\end{align}
Consequently, for any $\epsilon_{1}>0$ and sufficiently large $k$, we have
\begin{align}\label{equ.3.5}
h_{\mu_{k}}(\Phi_{k},\beta)\leq h_{\mu}(\Phi,\beta)+\epsilon_{1}.
\end{align}

Moreover, for the second term of \eqref{equ.3.17}, it follows from Theorem \ref{lem.3.5} and Lemma \ref{lem.3.6} that
\begin{align}\label{equ.3.4}
r(\Phi_{k},\epsilon,\Lambda)\leq\bar{G}_{\textup{loc}}(\Phi_{k},\epsilon)\leq \dfrac{d}{r}R(\Phi_{k})=0.
\end{align}
Combining the above estimates, we have
\[
h_{\mu_{k}}(\Phi_{k})\leq h_{\mu_{k}}(\Phi_{k},\beta) \leq h_{\mu}(\Phi,\beta)+\epsilon_{1} \leq h_{\mu}(\Phi)+\epsilon_{1},
\]
where the first inequality follows from \eqref{equ.3.17} and \eqref{equ.3.4}, and the second from \eqref{equ.3.5}. The proof is complete.
\end{proof}
\begin{remark}\rm
Consider the following $d$-dimensional SDEs
\begin{equation}\label{m-dim}
\d X=A(X)\d t+\bar{B}(X)\d \bar{W}, \qquad X(0)\in\R^{d}
\tag{E.1}
\end{equation}
and
\begin{equation}\label{m-dimk}
\d X=A_{k}(X)\d t+\bar{B}_{k}(X)\d \bar{W}, \qquad X(0)\in\R^{d},\qquad k\in\N^{+},
\tag{E.2}
\end{equation}
where $A,A_{k}:\R^{d}\rightarrow\R^{d}$ and  $B,\bar{B}_{k}:\R^{d}\rightarrow\R^{d\times m}$ for some integer $m\geq 1$, and $\bar{W}:=(W_{i})_{1\leq i\leq m}$ is an $m$-dimensional Brownian motion. The conclusions  regarding equations \eqref{equ.1.1} and \eqref{equ.1.2} in this section apply to equations \eqref{m-dim} and \eqref{m-dimk} as well; the proof is just a modification of notations.
\end{remark}

\section{Applications}\label{sec.4}
In this section, we give two applications. From Theorem \ref{the.3.1}, note that to establish the upper semicontinuity of entropy, the key steps are to verify that $\mu_{k}$ converges to $\mu$ under the $\text{weak}^{\ast}$-topology and that condition \eqref{pro.3.1} holds.  Firstly, we give a sufficient condition to guarantee the continuity of invariant measures with respect to coefficients---the Lyapunov function condition. Furthermore, for the SDEs driven by additive noise with smooth coefficients, the condition \eqref{pro.3.1} naturally holds, so the upper semicontinuity of the measure-theoretic entropy
follows.

\subsection{Continuity under the Lyapunov function condition}
Now we introduce the notion of uniform Lyapunov functions from Huang et al \cite{Huang1} and \cite{Huang2}.
A function $V:\R^{d}\to\R$ is called \textit{Lyapunov function} with respect to equation \eqref{m-dim} if $V$ is a nonnegative $C^{2}$ function such that
$$
\lim_{x\rightarrow\infty}V(x)=+\infty,
$$
and for a bounded open set $U\subset \R^{d}$ with regular boundary, there exists some constant $\gamma>0$ satisfying
$$
\mathcal{L}V (x)\leq-\gamma \quad \hbox{for } x\in \R^{d}\setminus U,
$$
where $\mathcal{L}$ is the adjoint Fokker--Planck operator generated by equation \eqref{m-dim}:
\begin{align}
\mathcal{L}V :=\sum_{i=1}^d a^i \frac{\partial V}{\partial x_i}+\frac{1}{2} \sum_{i,j=1}^d b^{ij} \frac{\partial^2 V}{\partial x_i \partial x_j}\notag
\end{align}
with $A=(a^{1},\cdots,a^{d})^{\intercal}$ and $\bar{B}\bar{B}^{\intercal}=(b^{ij})_{1\leq i,j\leq d}$. A function is a \textit{uniform Lyapunov function} on $\R^{d}$ with respect to equations \eqref{m-dimk} if it is a Lyapunov function with respect to each equation \eqref{m-dimk}, and $\gamma$ and $U$ are independent of $k$. Then we have the following continuity result on invariant measures.

\begin{prop}\label{pro.4.2}
For equations \eqref{m-dim} and \eqref{m-dimk}, assume that there exist unique solutions $\Phi$ and $\Phi_{k}$, respectively. Suppose that there is a uniform Lyapunov function for \eqref{m-dim} and \eqref{m-dimk}, and that $\bar{B}(x)\bar{B}^{\intercal}(x), \bar{B}_{k}(x)\bar{B}_{k}^{\intercal}(x)$  are invertible whenever $x\in U$.  Then there exist unique invariant measures $\mu$ and $\mu_{k}$ for equations \eqref{m-dim} and \eqref{m-dimk}, respectively. Furthermore, if
$$
\lim\limits_{k\rightarrow\infty}\vert A_{k}(x)-A(x)\vert+\vert \bar{B}_{k}(x)-\bar{B}(x)\vert=0\quad \text{point-wise},
$$
then $\mu_{k}$ converges to $\mu$
under the $\text{weak}^{\ast}$-topology.
\end{prop}

\begin{proof}
Based on Huang et al \cite{Huang1}, there exist unique invariant measures $\mu$ and $\mu_{k}$ for equations \eqref{m-dim} and \eqref{m-dimk}, respectively. Furthermore, through Ma and Liu \cite{LM}, we get that $\mu_{k}$ converges to $\mu$ under the $\text{weak}^{\ast}$-topology.
\end{proof}

According to Theorem \ref{the.3.1} and Proposition \ref{pro.4.2}, we get the upper semicontinuity of entropy under the Lyapunov function condition. For the matrix-valued function $\bar{B}=(b_{1},\cdots,b_{m})$, denote $\Vert \bar{B}\Vert_{C^{0}}:=\max_{1\leq i\leq m}\Vert b_{i}\Vert_{C^{0}}$.
\begin{theorem}\label{the.4.2}
For equations \eqref{m-dim} and \eqref{m-dimk}, let $A,A_{k},\bar{B},\bar{B}_{k}$ be $C^{\infty}$ functions with bounded derivatives of all orders. Suppose that there is a uniform Lyapunov function for \eqref{m-dim} and \eqref{m-dimk}, and that $\bar{B}(x)\bar{B}^{\intercal}(x), \bar{B}_{k}(x)\bar{B}_{k}^{\intercal}(x)$  are invertible whenever $x\in U$. If for every $k\in \N^{+}$ and any integer $1\leq r<\infty$,
\begin{align}
\int_{\Omega}\ln^{+}\Vert D\Phi_{k}(\omega)\Vert_{C^{r-1}}\d\P(\omega)<+\infty\notag
\end{align}
and
$$
\lim\limits_{k\rightarrow\infty}\Vert A_{k}-A\Vert_{C^{0}}+\Vert \bar{B}_{k}-\bar{B}\Vert_{C^{0}}=0,
$$
then the measure-theoretic entropy is upper semicontinuous:
\begin{equation}
\limsup\limits_{k\rightarrow\infty}h_{\mu_{k}}(\Phi_{k})\leq h_{\mu}(\Phi).\notag
\end{equation}
\end{theorem}

\subsection{Continuity for SDEs with additive noise}
Consider the following SDEs driven by additive noise:
\begin{align}\label{equ.4.2}
\d X=A(X)\d t+\bar{B}\d \bar{W}, \quad X(0)=x\in \R^{d}
\end{align}
and
\begin{align}\label{equ.4.4}
\d X=A_{k}(X)\d t+\bar{B}_{k}\d \bar{W}, \quad X_{k}(0)=x\in \R^{d},\quad k\in\N^{+},
\end{align}
where $A,A_{k}:\R^{d}\rightarrow\R^{d}$, $\bar{B},\bar{B}_{k}\in \R^{d\times m}$ are constant matrices, and $\bar{W}:=(W_{i})_{1\leq i\leq m}$ is an $m$-dimensional Brownian motion. As described in Proposition \ref{lem.2.1}, the derivative $D \Phi(x,t)$ of solution $\Phi(x,t)$ to \eqref{equ.4.2} satisfies the following linearized equation:
\begin{equation}
D\Phi(x,t)=I+\int_{0}^{t}DA(\Phi(x,s))D\Phi(x,s)\d s,\notag
\end{equation}
where $I$ is the identity matrix.
Similarly, the $n$-th derivative satisfies the following equation
\

\begin{align}
D^{n}\Phi(x,t) =I+\int^{t}_{0}\big( DA(\Phi(x,s)) D^n \Phi(x,s) + P_n(D\Phi(x,s), \dots, D^{n-1}\Phi(x,s))\big)\d t,\notag
\end{align}
where $P_n$ is a polynomial in $D\Phi, \dots, D^{n-1}\Phi$
with coefficients depending on $D^2A, \dots, D^nA$.

Next, we will show that the solution of equation \eqref{equ.4.2}/\eqref{equ.4.4} satisfies condition \eqref{pro.3.1} under the $C^{\infty}$
regularity assumption. This indicates that the upper semicontinuity of the entropy for equation \eqref{equ.4.2} can be verified directly from the coefficients.

\begin{lemma}
For equation \eqref{equ.4.2}, let $A$ be a $C^{\infty}$ function on $\R^{d}$ with bounded derivatives of all orders. Then, for any $r\geq 1$, we have
\begin{align}
\int_{\Omega}\Vert D\Phi(\omega)\Vert_{C^{r-1}}\d\P(\omega)<+\infty.\notag
\end{align}
\end{lemma}
\begin{proof}
We only prove the case $r=1$. The case $r>1$ can be obtained similarly.
Notice that
\begin{align}
&\int_{\Omega}\sup_{x\in\R^{d}}\vert D\Phi(x,1,\omega)\vert\d\P(\omega)\notag\\
\leq &1+\int_{\Omega}\sup_{x\in\R^{d}}\vert\int_{0}^{1}DA(\Phi(x,s,\omega))D\Phi(x,s,\omega)\d s\vert\d\P(\omega)\notag\\
\leq &1+K\int_{0}^{1}\int_{\Omega}\sup_{x\in\R^{d}}\vert D\Phi(x,s,\omega)\vert\d\P(\omega)\d s,\notag
\end{align}
where $K:=\sup_{x\in\R^{d}}\vert DA(x)\vert$.
Then by the Gronwall inequality, we get
\begin{align}
\int_{\Omega}\sup_{x\in\R^{d}}\vert D\Phi(x,1,\omega)\vert\d\P(\omega)\leq e^{K}.\notag
\end{align}
The proof is complete.
\end{proof}

\begin{theorem}\label{the.4.1}
For equations \eqref{equ.4.2} and \eqref{equ.4.4}, assume that $A,A_{k}$ are $C^{\infty}$ functions on $\R^{d}$ with bounded derivatives of all orders. Suppose that there is a uniform Lyapunov function for \eqref{equ.4.2} and \eqref{equ.4.4}, and $\bar{B}\bar{B}^{\intercal}, \bar{B}_{k}\bar{B}_{k}^{\intercal}$ are invertible. If
\begin{align}
\lim\limits_{k\rightarrow\infty}\Vert A_{k}-A\Vert_{C^{0}}+\vert \bar{B}_{k}-\bar{B}\vert=0,\notag
\end{align}
then the measure-theoretic entropy is upper semicontinuous:
\begin{align}
\limsup\limits_{k\rightarrow\infty}h_{\mu_{k}}(\Phi_{k})\leq h_{\mu}(\Phi).\notag
\end{align}
\end{theorem}

\section*{Acknowledgements}
This work is supported by National Key R\&D Program of China (No. 2023YFA1009200), NSFC (Grants 12531009 and 11925102), and Liaoning Revitalization Talents Program (Grant XLYC2202042).

\end{document}